\documentclass{amsart}
\usepackage{amssymb}
\usepackage{mathtools}
\usepackage{mathrsfs}
\usepackage{amscd} 
\usepackage[normalem]{ulem}

\usepackage{cancel}

\usepackage{color}

\usepackage{hyperref}
\usepackage{graphicx}
\usepackage[curve]{xypic}
\usepackage[usenames,dvipsnames,x11names]{xcolor}
\usepackage{tikz}
\usetikzlibrary{arrows,snakes,shapes,shadows}
\usepackage{verbatim}

\newbox\mybox
\def\overtag#1#2#3{\setbox\mybox\hbox{$#1$}\hbox to
  0pt{\vbox to 0pt{\vglue-#3\vglue-\ht\mybox\hbox to \wd\mybox
      {\hss$\ss#2$\hss}\vss}\hss}\box\mybox}
\def\undertag#1#2#3{\setbox\mybox\hbox{$#1$}\hbox to 0pt{\vbox to
    0pt{\vglue#3\vglue\ht\mybox\hbox to \wd\mybox
      {\hss$\ss#2$\hss}\vss}\hss}\box\mybox}
\def\lefttag#1#2#3{\hbox to 0pt{\vbox to 0pt{\vglue -6pt\hbox to
      0pt{\hss$\ss#2$\hskip#3}\vss}}#1}
\def\righttag#1#2#3{\hbox to 0pt{\vbox to 0pt{\vglue -6pt\hbox to
      0pt{\hskip#3$\ss#2$\hss}\vss}}#1}
\let\ss\scriptstyle
\def\splicediag#1#2{\xymatrix@R=#1pt@C=#2pt@M=0pt@W=0pt@H=0pt}
\def\Dot{\lower.2pc\hbox to 2pt{\hss$\bullet$\hss}}
\def\Circ{\lower.2pc\hbox to 2pt{\hss$\circ$\hss}}
\def\Vdots{\raise5pt\hbox{$\vdots$}}
\newcommand\lineto{\ar@{-}}
\newcommand\dashto{\ar@{--}}
\newcommand\dotto{\ar@{.}}

\usepackage{color}

\newcommand{\al}{{\alpha}}

\newcommand{\p}{{\partial}}
\newcommand{\s}{{\sigma}}

\newcommand{\IN}{{\mathbb N}}

\newcommand{\IR}{{\mathbb R}}
\newcommand{\IC}{{\mathbb C}}

\newcommand{\cC}{{\mathcal C}}

\newcommand{\cO}{{\mathcal O}}

\newtheorem{theorem}{Theorem}[section]
\newtheorem{proposition}[theorem]{Proposition}
\newtheorem*{theorem*}{Theorem}

\newtheorem{lemma}[theorem]{Lemma}

\theoremstyle{definition}

\newtheorem{remark}[theorem]{Remark}

\newtheorem*{remark*}{Remark}
\newtheorem*{claim*}{Claim}

\newtheorem{definition}[theorem]{Definition}

\newtheorem{proposition-definition}[theorem]{Proposition-definition}

\begin{document}
\title[Sharp estimates for blowing down Denjoy-Carleman functions]{Sharp estimates for blowing down functions in a Denjoy-Carleman class}

\author[A.~Belotto da Silva]{Andr\'{e} Belotto da Silva}
\author[E.~Bierstone]{Edward Bierstone}
\author[A.~Kiro]{Avner Kiro}

\address[A.~Belotto da Silva]{Universit\'e Aix-Marseille, Institut de Math\'ematiques de Marseille
(UMR CNRS 7373), Centre de Math\'ematiques et Informatique, 39 rue F. Joliot Curie, 
13013 Marseille, France}
\email[A.~Belotto da Silva]{andre-ricardo.belotto-da-silva@univ-amu.fr}
\address[E.~Bierstone]{University of Toronto, Department of Mathematics, 40 St. George Street, Toronto, ON, Canada M5S 2E4}
\email[E.~Bierstone]{bierston@math.utoronto.ca}
\thanks{Research supported by NSERC Discovery Grant RGPIN-2017-06537 (Bierstone)}
\address[A.~Kiro]{Weizmann Institute of Science,Faculty of Mathematics and Computer Science, Department of Mathematics, Rehovot, Israel}
\email[A.~Kiro]{avner-ephraiem.kiro@weizmann.ac.il}
\thanks{Research supported in part by ERC Advanced Grant 692616, ISF Grant 382/15, and by the European Research Council (ERC) under the European Union’s Horizon 2020 research and innovation programme, grant agreement No 802107 (Kiro).}
\thanks{{The authors gratefully acknowledge the referee's careful review and important comments.}}





\begin{abstract} 
If $F$ is a $C^{\infty}$ function whose composition $F \circ \sigma$ with a blowing-up $\s$ belongs to a Denjoy-Carleman class $C_M$, then $F$, in general, belongs to a larger class $C_{M^{(2)}}$; i.e., there is a loss of regularity. We show that this loss of regularity is sharp. In particular, loss of regularity of Denjoy-Carleman classes is intrinsic to arguments involving resolution of singularities.
\end{abstract}

\maketitle

\section{Introduction}
Quasianalytic Denjoy-Carleman classes $C_M$ go back to E. Borel \cite{Borel} and were characterized (following questions of Hadamard that arose from work of Holmgren on the heat equation) by the Denjoy-Carleman theorem \cite{Den,Carl}. We recall that $C_M$ is a class of $C^{\infty}$ functions with bounds on derivatives determined by a logarithmically convex
sequence $M =(M_k)_{k\in \mathbb{N}}$. Quasianalytic classes arise in model theory as the classes of 
$C^{\infty}$ functions that are definable in a given polynomially-bounded o-minimal structure \cite{Mil,RSW}. Relevant background on Denjoy-Carleman and quasianalytic classes is presented in Section \ref{sec:DC} below.

Given a (log convex) sequence $M= (M_k)_{k \in \mathbb{N}}$ and a positive integer $p$, 
we denote by {$M^{(p)}$ the shifted sequence $(M_{pk})_{k\in \mathbb{N}}$, and by $M^p$
the sequence of $p$th powers $(M_k^p)_{k\in \mathbb{N}}$. In general,} $C_M \subseteq C_{M^p} \subseteq C_{M^{(p)}}$; {moreover, $C_M = C_{M^{(p)}}$ if and only if} $C_M$ is the class of analytic functions. We recall that the shifted class $C_{M^{(2)}}$ is the smallest Denjoy-Carleman class containing all $g \in \cC^\infty([0,\infty))$ such that $g(t^2)\in C_M(\IR)$ \cite[Rmk.\,6.2]{Nelim}, {cf. \cite{BKS2020}}.

The goal of this paper is to extend this {result}  to \emph{blowings-up}. {We can express the blowing up of the origin in the plane $\IR^2$ using polar coordinates, as the mapping $\sigma: \mathbb{R}^2 \to \mathbb{R}^2$ given by
\[
\s(r,\theta) = (r\cos \theta,\,r \sin \theta)
\]
(the universal covering of the standard blowing-up).}
{If $F \in C^{\infty}(\mathbb{R}^2)$ is a function} such that $F \circ \sigma \in C_M(\mathbb{R}^2)$, 
{then $F \in C_{M^{(2)}}(\mathbb{R}^2)$}, by \cite[Lemma 3.4]{BBB}. {We will} show that this estimate is sharp.

\begin{theorem}[Sharp estimate {for blowing up}]\label{thm:Main}
Let $C_M$ {be} a Denjoy-Carleman class closed {under} differentiation, such that $C_{M^2}=C_{M^{(2)}}$. {Then, for every Denjoy-Carleman} class $C_{N} \subsetneq C_{M^{(2)}}$, there exists 
$F \in C_{M^{(2)}}(\mathbb{R}^2) \setminus C_N(\mathbb{R}^2)$ such that $F\circ \sigma \in C_M(\mathbb{R}^2)$.
\end{theorem}

{Moreover}, under the hypothesis that the class $C_M$ is quasianalytic, our techniques {provide} 
the following result.

\begin{theorem}[Quasianalytic estimate for  blowing up]\label{thm:Main2}
Let $C_M$ be a quasianalytic Denjoy-Carleman class closed {under} differentiation. For every
{Denjoy-Carleman}
class $C_{N} \subsetneq C_{M^{(2)}}$ such that $\lim (N_k/M_{2k})^{1/k} = 0$, there exists $F \in C_{M^{(2)}}(\mathbb{R}^2) \setminus C_N(\mathbb{R}^2)$ such that $F\circ \sigma \in C_M(\mathbb{R}^2)
$. 

In particular, if $C_M$ properly contains the analytic functions, {then} there exists $F \in C_{M^{(2)}}(\mathbb{R}^2) \setminus C_M(\mathbb{R}^2)$ such that $F \circ \sigma \in C_M(\mathbb{R}^2)$.
\end{theorem}

{Note that, if $F\circ \sigma\in C_M$, then $F$ restricts to a function of class $C_M$ on any 
nonsingular $C_M$ curve (since $\s$ restricts to an isomorphism from a lifting of any such curve);}
in particular, $F$ is formally {of class} $C_M$ at every point. Therefore the loss of regularity (i.e., 
the fact that $F$ belongs to the Denjoy-Carleman class $C_{M^2}$ or $C_{M^{(2)}}$ instead of $C_{M}$,
{according to} Theorems \ref{thm:Main} or \ref{thm:Main2}) is {intrinsically a phenomenon} of two or more variables---it cannot be observed by sampling the function on a nonsingular curve. {An
earlier example of a function of class $C^\infty$ (as opposed to {Denjoy-Carleman})
of more then one variable, which is not in $C_M$, although its restriction to any nonsingular $C_M$ curve is of class $C_M$,} was constructed by Jaffe \cite{Jaffe}. 
{In contrast, all explicit constructions of $C_M$ functions of which} we are aware (e.g., \cite{ABBNZ,Beurling,Borel,Carl,Jaffe,Mandel,NSV,Nelim,Ostrowski}; cf. the historical survey \cite{Bilodeau}), are either one variable constructions or of the form  $h_1\circ h_2$, where $h_1$ is a one variable $C_M$ function and $h_2$ is real-analytic.

The hypothesis that $C_{M^{(2)}} = C_{M^2}$ {in Theorem \ref{thm:Main} is mild in the sense that standard} examples of quasianalytic Denjoy-Carleman classes usually satisfy this hypothesis (e.g., $M_k=\log(k)^k$, $M_k= \log(\log(k))^k$, etc.); {moreover,} in quasianlaytic classes, the condition is verified for ``almost every" $k \in \mathbb{N}$ (see Lemma \ref{lem: extrareg2} for a precise result). It is nevertheless possible to construct examples of quasianalytic classes where the hypothesis is not satisfied.

\begin{proposition}[Classes such that $C_{M^2} \subsetneq C_{M^{(2)}}$]\label{prop:Difference}
There is a Denjoy-Carleman class $C_M$ closed under differentiation, such that $C_{M^{(2)}}$ is quasianalytic and $C_{M^2} \subsetneq C_{M^{(2)}}$.
\end{proposition}

The paper is organized as follows. In Section \ref{sec:DC}, we provide all necessary background on quasianalytic Denjoy-Carleman classes $C_M$, including new results about shifted classes (Lemmas \ref{lem:CharacterizationCM2} and \ref{lem: extrareg2}). In Section \ref{sec:Construction}, we explicitly construct the main examples and we prove Theorems \ref{thm:Main} and \ref{thm:Main2}. The construction relies on two preliminary steps (see \S\S\ref{ssec:Brick}, \ref{ssec:Build}), where we provide tight estimates for the ``bricks" and ``building blocks" of the main construction. The main example is constructed in \S\ref{ssec:Flat}, and Theorems \ref{thm:Main} and \ref{thm:Main2} are proved in \S\ref{ssec:proofMain}. Finally, we prove Proposition \ref{prop:Difference} in Section \ref{sec:ProofProp}.

\subsection{Algebraic properties of quasianalytic classes}
A Denjoy-Carleman class $C_M$ which is closed {under} differentiation and quasianalytic admits resolution of singularities, according to \cite{BMinv,BMselecta}. Resolution of singularities is a powerful technique which, roughly {speaking}, associates to a function $f \in C_M(U)$ a finite composite of blowings-up $\sigma : V \to U $ such that $f\circ \sigma$ is everywhere locally a monomial times a unit. This technique was explored in the last twenty years in order to study geometric and algebraic properties of quasianalytic classes {\cite{BMselecta, RSW, Nelim, RS, BBB, BBC, BBmon}}. 

{Solutions of problems on Denjoy-Carleman classes $C_M$ using resolution of singularities, in general lead to \emph{loss of regularity}, in the sense that  a problem involving data of class $C_M$ has
solutions in a shifted class $C_{M^{(p)}}$, for certain $p \in \mathbb{N}$ \cite{BBB,BBC}. Theorem \ref{thm:Main} shows that loss of regularity is an essentially unavoidable limitation of the technique of resolution of singularities. It seems important, therefore, to understand whether loss of regularity is a limitation only of the technique, or is intrinsic to geometric questions on Denjoy-Carleman classes.}

For example, resolution of singularities is used in \cite[Proposition 4.9]{BBB} to show that that principal ideals in a local ring of functions in a quasianalytic Denjoy-Carleman class $C_M$ are \emph{closed}, modulo loss of regularity;
i.e., given $f,\,g \in C_M(U)$, where $U$ is a neighbourhood of $0\in\IR^n$, if $f$ formally divides $g$ at $0$, 
then there is a function $h \in C_{M^{(p)}}(V)$ in a neighbourhood $V$ of $0$, such that $g =f \cdot h$, 
where $p \in \mathbb{N}$ depends only on $f$. If loss of regularity were necessary in this problem, this would 
imply that local rings in a quasianalytic class $C_M$ are not, in general, Noetherian. The question of
Noetherianity in quasianaltyic geometry goes back to \cite{Child} (see also \cite{Thilliez}).


\section{Preliminaries}\label{sec:DC}
\subsection{Denjoy-Carleman classes}\label{subsec:DC}
We use standard multiindex notation: Let $\IN$ denote the nonnegative integers. If $\al = (\al_1,\ldots,\al_n) \in
\IN^n$, we write $|\al| := \al_1 +\cdots +\al_n$, $\al! := \al_1!\cdots\al_n!$, $x^\al := x_1^{\al_1}\cdots x_n^{\al_n}$,
and $\p^{|\al|} / \p x^{\al} := \p^{\al_1 +\cdots +\al_n} / \p x_1^{\al_1}\cdots \p x_n^{\al_n}$. {We also write
$\p^\al := \p^{|\al|} / \p x^{\al}$.} 

\begin{definition}[Denjoy-Carleman classes]\label{def:DC}
Let $M = (M_k)_{k\in \IN}$ denote a sequence of positive real numbers which is \emph{logarithmically
convex}; i.e., the sequence $(M_{k+1} / M_k)$ is nondecreasing.
A \emph{Denjoy-Carleman
class} $C = C_M$ is a class of $\cC^\infty$ functions determined by the following condition: A function 
$f \in \cC^\infty(U)$ (where $U$ is open in $\IR^n$) is of class $C_M$ if, for every compact subset $K$ of $U$,
there exist constants $A,\,B > 0$ such that
\begin{equation}\label{eq:DC}
\left|\frac{\p^{|\al|}f}{\p x^{\al}}\right| \leq A B^{|\al|} \al! M_{|\al|}
\end{equation}
on $K$, for every $\al \in \IN^n$.
\end{definition}

The logarithmic convexity assumption implies that $M_jM_k \leq M_0M_{j+k}$, for all $j,k$, and that the sequence $((M_k/M_0)^{1/k})$ is nondecreasing. The first of these conditions guarantees that $C_M(U)$ is a ring, and  the second that $C_M(U)$ contains the ring $\cO(U)$ of real-analytic functions on $U$, for every open $U\subset \IR^n$. (If $M_k=1$, for all $k$, then $C_M = \cO$.). A  Denjoy-Carleman class $C_M$ is closed under composition, {by Roumieu \cite{Rou}}.

{We will always assume that $M$ satisfies the additional assumption,}
\begin{equation}\label{eq:ClosedByDiff}
\displaystyle{\sup \left(\frac{M_{k+1}}{M_k}\right)^{1/k} < \infty};
\end{equation}
{the latter implies} that $C_M$ is closed under differentiation. The converse of this statement is due to S. Mandelbrojt \cite{Mandel}. In a Denjoy-Carleman class $C_M$, closure under differentiation is equivalent to closure under division by a coordinate; i.e., if $f \in C_M(U)$ and
$$
f(x_1,\dots, x_{i-1}, a, x_{i+1},\ldots, x_n) = 0,
$$
where $a \in \IR$,  then $f(x) = (x_i - a)h(x),$ where $h \in C_M(U)$. 

Finally, closure under differentiation implies closure under inverse {(Komatsu \cite{Kom}; see \cite{BMselecta} for a simple proof)}. More precisely, let $\varphi : U \to V$ denote a $C_M$-mapping between open subsets $U$, $V$ of $\IR^n$. Let $a \in  U$ and suppose that the Jacobian matrix
$(\partial \varphi/\partial x)(a)$ is invertible. Then there are neighbourhoods $U'$ of $a$ and $V'$ of $b := \varphi(a)$, and a $C_M$-mapping  $\psi: V' \to U'$ such that
$\psi(b) = a$ and $\psi\circ \varphi$  is the identity mapping of
$U '$.

\subsection{Comparison between Denjoy-Carleman classes.}\label{ssec:Comparison}
The following {criteria are due} to Cartan and Mandelbrojt (see \cite[Thm.\,XI]{Mandel}). 
If $C_M$, $C_N$ are Denjoy-Carleman classes closed under differentiation, then: 

\begin{itemize}
\item[(a)] $C_N(U) \subseteq C_M(U)$, for all $U$, if and only if 
\[
\sup_{k \in \mathbb{N}} \left(N_k /M_k\right)^{1/k} < \infty;
\]
\item[(b)] $C_N(U) \subsetneq C_M(U)$, for all $U$,
if and only if 
\[
\sup_{k \in \mathbb{N}} \left(N_k /M_k\right)^{1/k} < \infty \quad \text{ and }\quad \inf_{k \in \mathbb{N}} \left(N_k /M_k\right)^{1/k} =0.
\]
\end{itemize}
We note that for any  Denjoy-Carleman class $C_M$, there is a function in $C_M((0,1))$ which is nowhere in any smaller class {(by Jaffe \cite[Thm.\,1.1]{Jaffe}; E. Borel constructed a function in $C_M((0,1))$
that is nowhere analytic \cite[Chapt.\,2]{BorelPhD})}.

\subsection{Shifted Denjoy-Carleman classes}\label{subsec:shift}
Given $M = (M_j)_{j\in \IN}$ and a positive integer $p$, let $M^{(p)}$ denote the sequence $M^{(p)}_j := M_{pj}$. If $M$ is logarithmically convex, then $M^{(p)}$ is logarithmically convex:
$$
\frac{M_{kp}}{M_{(k-1)p}} = \frac{M_{kp}}{M_{kp-1}} \cdots \frac{M_{kp -p+1}}{M_{kp - p}} \leq
\frac{M_{kp+p}}{M_{kp+p-1}} \cdots \frac{M_{kp+1}}{M_{kp}} = \frac{M_{(k+1)p}}{M_{kp}}.
$$
Therefore, if $C_M$ is a Denjoy-Carleman class, then so is $C_{M^{(p)}}$. Clearly,
$C_M \subseteqq C_{M^{(p)}}$. Moreover, if $C_M$ satisfies assumption \eqref{eq:ClosedByDiff}, then the same is true for $C_{M^{(p)}}$. We recall that $C_{M^{(2)}}$ is the smallest Denjoy-Carleman class containing all $g \in \cC^\infty(\IR)$ such that $g(t^2)\in C_M(\IR)$ \cite[Rmk.\,6.2]{Nelim}.

In {Lemma \ref{lem:CharacterizationCM2} following, we characterize} sequences $M$ such that the shifted class $C_{M^{(2)}}$ {equals} $C_{M^2}$, where $M^2 := (M_k^2)_{k \in \mathbb{N}}$. {By log-convexity,} $C_{M^2} \subset C_{M^{(2)}}$. In order to obtain equality, we need the additional condition,
\[
\inf_{k  \in \mathbb{N}} \left\{  \left( \frac{M_{k}^2}{M_{2k}} \right)^{1/k} \right\} > 0\,.
\]
The latter is a regularity condition that complements log-convexity. More precisely, by log-convexity,
\[ 
\frac{M_k^2}{M_{2k}}=M_k\cdot\frac{M_k}{M_{k+1}}\cdots\frac{M_{2k-1}}{M_{2k}}\le M_k\left(\frac{M_k}{M_{k+1}}\right)^k,
\] 
so that
\begin{equation}\label{eq:ComparisonConditions}
\left(\frac{M_k^2}{M_{2k}}\right)^{1/k}\leq\, \frac{M_k^{1+1/k}}{M_{k+1}}\,, \quad \text{ for all }\, k\in \mathbb{N}.
\end{equation}
{The right hand side of \eqref{eq:ComparisonConditions} can be used to characterize} sequences $M$ such that $C_{M^2} \subset C_{M^{(2)}}$:

\begin{lemma}\label{lem:CharacterizationCM2}
Let $M$ be a log convex sequence with $M_0=1$. Then $C_{M^2}= C_{M^{(2)}}$ if and only if
\[
 \inf_{k \in \mathbb{N}} \left\{ \frac{M_{k}^{{1+ 1/k}}}{M_{k+1}} \right\}> 0 
\]
\end{lemma}
\begin{proof}
{The ``{only if}'' direction} is immediate from inequality \eqref{eq:ComparisonConditions}. To prove the converse, note that there exists a constant $a>0$ such that
\[
M_k \geq a^{k/{ (k+1)}} \cdot M_{k+1}^{k/{(k+1)}}, \quad \text{ for all }\, k \in \mathbb{N}.
\]
{By applying this inequality $k$ times, it therefore follows that} 
\[
M_k \geq \prod_{j=1}^k a^{k/{(k+j)}} \cdot M_{2k}^{1/2}, \quad \text{ for all }\, k \in \mathbb{N},
\]
and, {using Jensen's} inequality, we get
\[
 \left( \frac{M_{k}^2}{M_{2k}} \right)^{1/k} \geq\, \prod_{j=1}^k a^{2/{(k+j)}}\, \geq\, \exp\left( \log(a^2) \sum_{j=1}^k \frac{1}{k+j} \right) \geq a^8,
\]
proving the converse.
\end{proof}

\subsection{Quasianalytic Denjoy-Carleman classes.} We say that a Denjoy-Carleman class $C_M$ is {\emph{quasianalytic} if is satisfies the following condition: if} $f \in C_M(U)$ has Taylor expansion zero at $a \in U$, then $f$ is identically zero near $a$. According to the Denjoy-Carleman theorem \cite[Thm.\,1.3.8]{Horm}, the class $C_M$ is quasianalytic if and only if
\begin{equation}\label{eq:Quasianalytic}
\displaystyle{\sum_{k=0}^\infty\frac{M_k}{(k+1)M_{k+1}} = \infty} \quad \text{ or, equivalently, } \quad \displaystyle{\sum_{k=0}^\infty\frac{1}{(k+1)M_{k}^{1/k}} = \infty}.
\end{equation}
{(Equivalence of the latter two criteria follows from log-convexity} of $M$.) In the following, a quasianalytic Denjoy-Carleman class $C_M$ is always assumed to be closed under differentiation. Quasianalytic Denjoy-Carleman classes $C_M$ admit resolution of singularities {by sequences of blowings-up} \cite{BMinv,BMselecta}.

For a quasianalytic class $C_M$, the shifted sequence $M^{(2)}$ and the squared sequence $M^2$ are ``almost everywhere" comparable. In order to make this statement precise, we {first recall} the notion of density.

\begin{definition}[Density of {indices}]
Consider a set of positive integers $\Lambda$. For every $n\in \mathbb{N}$, let $A_\Lambda(n)$ denote the cardinality of the set $\{k\in \Lambda;\, k\leq n\}$. {The \emph{density} $\delta(\Lambda)$} of the set $\Lambda$ is defined as
\[
\delta(\Lambda):= \lim_{n\to \infty} \frac{A_\Lambda(n)}{n}.
\]
\end{definition}
\begin{remark}\label{rmk: harmonic sum postive density}
By the Abel summation formula \cite[Thm.\, 4.2]{Apostol},
\[ 
\sum_{k\in \Lambda,\; k\le n}\frac{1}{k}=\int_1^n \frac{A_\Lambda(x)}{x^2}dx+\frac{A_\Lambda(n)}{n}. \]
In particular, if  $\delta(\Lambda)>0$ then, for every sufficiently large $n$,
\[ 
\sum_{k\in \Lambda,\; k\le n}\frac{1}{k}>\frac{\delta(\Lambda)}{2}\log n. 
\]
\end{remark}

\begin{lemma}[Comparison between $C_{M^2}$ and $C_{M^{(2)}}$]\label{lem: extrareg2}
Let $M$ be a log convex sequence with $M_0=1$, such that $C_M$ is quasianalytic. Then,
for every $\varepsilon \in (0,1)$,
	\[ 
\delta(\Lambda_{\varepsilon})=0, \quad \text{ where } \quad	\Lambda_{\varepsilon}:=
	\left\{k:\in\mathbb{N}:\left(\frac{M_{k}^2}{M_{2k}}\right)^{{1/2k}}<1-\varepsilon\right\}.
	\]
\end{lemma}	
\begin{proof}
 Assume, for the sake of contradiction, that $\Lambda_\varepsilon$ does not have zero density. Thus there exists a subset $\widetilde{\Lambda}_\varepsilon\subseteq \Lambda_\varepsilon$ of positive density. Now, there exists an increasing sequence  $(\widetilde{a}_k)_{k \in \mathbb{N}} \subset \widetilde{\Lambda}_\varepsilon$ such that the sets $\{k\in \mathbb{N}; \, \widetilde{a}_{k+1}\leq 2^{n}\}$ have cardinality at least ${\delta(\widetilde{\Lambda}_\varepsilon)}2^{n-1} $, for all $n$ sufficiently large. 

This means that, for $n$ large enough, at most $1/\delta(\widetilde{\Lambda}_\varepsilon)$ subsequent intervals
$[2^{\ell-1},2^\ell)$, $l = n+1, n+2,\ldots$, do not contain points of $(\widetilde{a}_k)$. We can, therefore, choose a
subsequence $(a_k)_{k \in \mathbb{N}} \subset(\widetilde{a}_k)_{k \in \mathbb{N}}$ with the following properties:
$a_{k+1}\geq 2a_{k}$, each interval $[2^{\ell-1},2^\ell)$ contains at most one point $a_k$, and 
$\{k\in \mathbb{N}; \, a_{k+1}\leq 2^{n}\}$ has cardinality at least ${\delta(\widetilde{\Lambda}_\varepsilon)^2}n/4$, for
$n$ large enough.

Now consider $L_n=M_{n}^{1/n}$ and note that, since $L$ is increasing, for sufficiency large $n$,
\[ 
\begin{aligned}
\frac{1}{L_{2^n}}& \leq \prod_{a_{k+1} \leq 2^n} \frac{L_{a_k}}{L_{a_{k+1}}} \leq \prod_{a_{k+1} \leq 2^n} \frac{L_{a_k}}{L_{2a_k}} \leq
\prod_{a_{k+1}\le 2^{n}} (1-\varepsilon)\leq (1-\varepsilon)^{{\delta(\widetilde{\Lambda}_\varepsilon)^2 n/4}}.
\end{aligned}
\]
We conclude that
\[
\sum_{n\geq 1} \frac{1}{nL_n}=  \sum_{m\geq 0}  \sum_{n=2^{m}}^{2^{m+1}-1} \frac{1}{nL_{n}} \leq \sum_{m\geq 0} \frac{1}{L_{2^m}} \sum_{n=2^{m}}^{2^{m+1}-1} \frac{1}{n} \leq (\log 2)\sum_{m\geq 0} \frac{1}{L_{2^m}} <\infty
\]
(where we use {the fact that $1/L_{2^m}$} is bounded by a geometric series). This contradicts the Denjoy-Carleman criterion \eqref{eq:Quasianalytic} for quasianalyticity of $C_M$.
\end{proof}

\section{Construction of the main function}\label{sec:Construction}

\subsection{Ostrowski function.} The {\emph{Ostrowski function} $\varphi_M(r)$}  associated to the sequence $M$ is defined by
\[
\varphi_M(r):=\sup_{n\geq0} \frac{r^{n+2}}{M_n}
\]
{(see \cite{Ostrowski}). When there is no risk of confusion, we will denote $\varphi_M$ by $\varphi$. We recall the following 
well-known property of $\varphi_M$.}

\begin{lemma}[Property of Ostrowski function]\label{lem:PropertyVarphiI}
Let $M$ be a log convex sequence with $M_0=1$, and consider the sequence $m_k:=M_{k+1}/M_k$, 
$k\in \mathbb{Z}_{\geq 0}$. The function $\varphi = \varphi_M$ satisfies the property,
\begin{equation}\label{eq:PropertyVarphiI}
\frac{m_k^{2+k}}{\varphi(m_k)} = M_k\,.
\end{equation}
\end{lemma}
\begin{proof}
{Note that}
\[
\frac{m_k^{2+k}}{\varphi(m_j)} = \inf_{n\geq0} \frac{M_n}{m_k^{n-k}} = \inf_{n\geq0} \frac{M_n \cdot M_k^{n-k}}{M_{k+1}^{n-k}} \leq M_k\,.
\]
{By log-convexity of the sequence $(M_k)$ we get
\[
\frac{M_n \cdot M_k^{n-k}}{M_{k+1}^{n-k}} \geq \frac{M_{k} \cdot M_{k+1}^{n-k}}{M_{k+1}^{n-k}} =M_k,
\quad \text{ if }\, n>k,
\]
and
\[
\frac{M_n \cdot M_{k+1}^{k-n}}{M_{k}^{k-n}} \geq \frac{M_{k} \cdot M_{k}^{{k-n}}}{M_{k}^{{k-n}}} =M_k,
\quad \text{ if }\, n<k.
\]
It follows that \eqref{eq:PropertyVarphiI} holds.}
\end{proof}

\subsection{Brick function and {\it a priori} estimates.}\label{ssec:Brick} It is important to get tight estimates
{on} derivatives of the ``bricks" used in this work. In this section, we derive {\it a priori} estimates via Cauchy estimates.

\begin{remark}[Cauchy estimate]
Let $f:U \subset \mathbb{C}^2 \to \mathbb{C} $ be an holomorphic function. Fix a point $x \in U$, and positive real numbers $r_1$ and $r_2$ such that the bi-disk 
\[
\mathbb{D}_2 := \{|z_i-x_i|\leq R_i,\, i=1,2\}
\]
is contained in $U$. The Cauchy estimate is given by:
\begin{equation}\label{eq:Cauchy}
\left|\frac{\partial^{\alpha}f(x)}{\alpha! }\right| \leq   \frac{ \max_{z\in \mathbb{D}_2}|f(z)|  }{R_1^{{\alpha_1}}R_2^{{\alpha_2}}}, \quad \text{ for all }\, \alpha \in \mathbb{Z}_{\geq 0}^2.
\end{equation}
\end{remark}

\noindent
Given constants $q \geq 1$, $m\ge 1$ and $0<\rho<1$, we consider the following {\emph{brick function}}:
\[
u_{q,m,\rho}(x) := \frac{\rho^2}{\rho^2+(x_1-\rho q)^2 + (m x_2)^2}.
\]
We {will use the following to estimate the derivatives of the brick function}.

\begin{lemma}[{\it A priori} estimate]\label{lem: triv der formula 2var}
{For every} $c \in \mathbb{R}_{>0}$ and $\alpha=(\alpha_1,\alpha_2)\in\mathbb{Z}_{\geq0}^2$, 
	\[  
	\left|\frac{1}{\alpha!} \cdot \partial^\alpha\left(\frac{1}{c+x^2_1+x_2^2}\right)\right| \leq \frac{8\cdot 8^{{|\alpha|}}}{(c+x_1^2+x_1^2)^{1+|\alpha|/2}}.  
	\]
\end{lemma}

\begin{proof}
{Consider} the holomorphic function $f(z) =(c+z^2_1+z_2^2)^{-1}$ and a point $x\in \mathbb{R}^2$. 
{Let $\mathbb{D}_2$ now denote} the bi-disk
\[
\mathbb{D}_2:=\left\{(z_1,z_2)\in \mathbb{C}^2\;:\; |z_i-x_i|\leq {8}^{-1}\sqrt{c+x_1^2+x_2^2},\; i=1,2  \right\},
\]
and note that $|x_1|+|x_2| \leq \sqrt{2}\sqrt{c+x_1^2+x_2^2}$. It follows from the triangle inequality that
\[
\begin{aligned}
\max_{z\in \mathbb{D}_2}|f(z)| &= \max_{z\in \mathbb{D}_2} \left(c+ (z_1- x_1+x_1)^2+(z_2- x_2+x_2)^2\right)^{-1}\\
&\leq \max_{z\in \mathbb{D}_2} \left( c+ x_1^2 + x_2^2 - 2\, |x_1|\cdot |z_1- x_1| - 2\, |x_2| \cdot|z_2- x_2|\right.\\
& \quad \quad \quad \quad \quad \quad \quad \quad \quad \quad \quad \quad \quad \quad \quad \quad 
\left.- |z_1- x_1|^2-|z_2- x_2|^2\right)^{-1}\\
&\leq  \left( {8}^{-1}(c+ x_1^2 + x_2^2)\right)^{-1} = {8} f(x).
\end{aligned}
\]
Using the Cauchy estimate \eqref{eq:Cauchy}, we get
\[  
\left|\frac{\partial^\alpha f(x)}{\alpha !}\right|\,\leq\, \max_{z\in \mathbb{D}_2}|f(z)| \cdot \left(\frac{\sqrt{x_1^2+x_2^2+c}}{{8}}\right)^{-{|\alpha|}} \leq\, \frac{{8} \cdot {8}^{{|\alpha|}}}{\left(x_1^2+x_2^2+c\right)^{1+|\alpha|/2}}. \qedhere
\]
\end{proof}	

\begin{remark}[A priori estimate on the brick function]
{It follows from Lemma \ref{lem: triv der formula 2var} that, for every $q \geq 1$, $m\ge 1$, $0<\rho<1$ and
$\alpha=(\alpha_1,\alpha_2)\in\mathbb{Z}_{\geq0}^2$,}
\[
\left|\frac{\partial^\alpha u_{q,m,\rho}(x)}{\alpha!} \right| \leq  \rho^2 \cdot m^{\alpha_2} \cdot {8^{|\alpha|+{1}}} \cdot \left(\frac{u_{q,m,\rho}(x)}{\rho^2}\right)^{1+|\alpha|/2}.
\]
\end{remark}

\smallskip
Now let us consider the blowing-up $\s: \IR^2 \to \IR^2$ (in polar coordinates),
\[
\sigma(r,\theta) = (r \cos\theta,\,r \sin\theta),
\]
and the composite {brick function}
\[
\begin{aligned}
v_{q,m,\rho}(r,\theta) &:= u_{q,m,\rho} \circ \sigma (r,\theta) = \frac{\rho^2}{\rho^2+ (r \cos\theta -\rho q)^2 +  (m\cdot r \sin\theta)^2}\,.
\end{aligned}
\]

\begin{lemma}[{\it A priori} estimate on the blowing-up of the brick function]\label{lem: base function der in polar coor}
There is a (universal) constant $C>0$ such that for all real numbers $q\geq 1$, $m\ge 1$, $0<\rho<1$, 
{as well as all} $\alpha\in\mathbb{Z}_{\ge 0}^2$, $r\ge 0$ and $|\theta|\leq\pi$, 
\[
	\left|\frac{\partial^\alpha v_{q,m,\rho}(r,\theta)}{\alpha!}\right| \leq  m^{|\alpha|}\left( 1 + q\rho \right)^{\alpha_2}  \cdot C^{|\alpha|+1}.
	\] 
\end{lemma}

\begin{proof}
Fix $r\geq 0$ and $\theta \in [-\pi,\pi]$, and let $x_1 = r\cos\theta -q\rho$, $x_2 = m r\sin\theta$. By Taylor approximation, there exists $A>1$ such that for every complex number $w\in \mathbb{C}$ {with} $|w-\theta|\leq 1$, we have
\begin{equation}\label{eq:TaylorSinCos}
\begin{aligned}
|\sin\theta-\sin w| &\leq A\cdot |\theta-w|,\\
|\cos\theta -\cos w| &\leq |\sin\theta|\cdot |\theta-w| + A\cdot |\theta-w|^2.
\end{aligned}
\end{equation}
We now consider the bi-disc $\widetilde{\mathbb{D}}_2:= \left\{(w_1,w_2)\in \mathbb{C}^2:   |w_1-r|\leq R_1,\, |w_2-\theta|\leq R_2 \right\}$, where the ratios are given by
\[
R_1 := \frac{({64}\rho)^{-1}\sqrt{\rho^2+ x_1^2 + x_2^2}}{|\cos\theta|+m|\sin\theta|}, \quad 
R_2 := \min\left\{1,\, \frac{({64}\rho)^{-1}\sqrt{\rho^2+ x_1^2 + x_2^2}}{m\cdot r \cdot A  +r \cdot |\sin\theta| +\sqrt{r\cdot A}}\right\},
\]
and we provide estimates for $R_1$, $R_2$ and $\max_{w\in \widetilde{\mathbb{D}}_2}|v_{q,m,\rho}(z)|$ which allow us to conclude using Cauchy estimates. Indeed, {since $m>1$ and} $A>1$,
\[
\begin{aligned}
\frac{1}{R_1} & = {64}\rho \cdot \frac{|\cos\theta|+m|\sin\theta|}{\sqrt{(r\cos\theta-q\rho)^2+m^2(r\sin\theta)^2+\rho^2}} \\ &\leq {64}\left(1 + \frac{\rho m|\sin\theta|}{\sqrt{(r\cos\theta-q\rho)^2+(r\sin\theta)^2+\rho^2}} \right) \leq {128} m;\\[1em]
\frac{1}{R_2} & = \max\left\{1,\, {64}\rho\cdot\frac{m\cdot r \cdot A  +r \cdot |\sin\theta| +\sqrt{r\cdot A}}{\sqrt{\rho^2+ x_1^2 + x_2^2}}\right\}\\
&\leq  {8^3} \cdot A \cdot m \cdot  \max\left\{1,\, \frac{ 3\rho r }{\sqrt{\rho^2+ x_1^2 + x_2^2}}\right\},
\end{aligned}
\]
where in the last inequality we used the fact that either $r<1$ and the max is smaller than $3$, or $r\ge 1$ and $\sqrt{r} \leq r$. {Since} $m>1$ and 
$\cos\theta \leq 1$, we get
\[
\begin{aligned}
\frac{1}{R_2} & \leq {8^3} \cdot A \cdot m \cdot  \max\left\{1,\, \frac{ 3\rho r }{\sqrt{\rho^2+ (r\cdot \cos(\theta) - q\rho)^2 + (mr\sin(\theta))^2}}\right\}\\
&\leq {8^3} \cdot A \cdot m \cdot  \max\left\{1,\, \frac{ 3\rho r }{\sqrt{\rho^2+ (r-q\rho)^2 }}\right\}\\
&\leq {8^4} \cdot A \cdot m \cdot  (1+ \rho q).
\end{aligned}
\]
{Let $z=(z_1,z_2)$ denote the complexification of $x$, i.e.,}
\[
z_1 = w_1 \cdot \cos w_2 - q, \quad z_2 = m \cdot w_1 \sin w_2
\]
{(where $w_1,w_2,z_1,z_2\in\IC$),} so that, using the triangle inequality, we get
\[ 
\begin{aligned}
	|z_1-x_1| &\leq r|\cos w_2-\cos\theta| + |w_1-r||\cos\theta|,\\
	\frac{|z_2-x_2|}{m}&\leq r|\sin  w_2-\sin \theta| +|w_1-r||\sin\theta|.\end{aligned}
\]
Now, {by the choices of $R_1$ and $R_2$,} for all $w \in \widetilde{\mathbb{D}}_2$, 
we have
\[ 
	\begin{aligned}
	|z_1-x_1|&\leq  |w_1-r||\cos\theta| + r\left(|\sin\theta|\cdot|w_2-\theta|+A\cdot \right |w_2-\theta|^2 )\\ 
	&\leq  ({64}\rho)^{-1}\sqrt{\rho^2+ x_1^2 + x_2^2},\\[.5em]
	|z_2-x_2|&\leq  m\cdot |w_1-r||\sin\theta| + r\cdot A\cdot m\cdot |w_2-\theta|\\
	&\leq ({64}\rho)^{-1}\sqrt{\rho^2+ x_1^2 + x_2^2},
	\end{aligned}
	\]
which implies that, {if $w \in \widetilde{\mathbb{D}}_2$, then} $z \in \mathbb{D}_2$, where $\mathbb{D}_2$ is the bi-disc {in} the proof of Lemma \ref{lem: triv der formula 2var}. If follows from the estimate in 
Lemma \ref{lem: triv der formula 2var} that
\[
\max_{w\in \widetilde{\mathbb{D}}_2}|v_{q,m,\rho}(z)| \leq \max_{z\in \mathbb{D}_2}|(1+ z_1^2 + z_2^2)^{-1}| \leq \frac{4}{1+ x_1^2+x_2^2} \leq 4.
\]
We conclude {using Cauchy estimates, taking} $C>0$ sufficiently big (for example, $C ={8^4} A$).
\end{proof}

\subsection{Building block and a priori estimates}\label{ssec:Build}
We start {with the existence of an important} function.

\begin{proposition}[Base function]\label{prop: h def}
Given a log convex sequence $M$ with $M_0=1$, there exists a 
``base function'' $h\in C^\infty (\mathbb{R}^2)$ with the following properties:
\begin{enumerate}
	\item[(i)] 	  for every $x=(x_1,x_2)\in\mathbb{R}^2$ and  $\alpha\in\mathbb{Z}_{\geq 0}^2$,
	\[
	\left|\partial^\alpha h(x)\right|\leq \frac{{64} \cdot {8}^{|\alpha|+1} \alpha!}{(1+ x_1^2+x_2^2)^{1+|\alpha|/2}}  \cdot M_{\alpha_2},
	\]
	\item[(ii)] for any $n \in \mathbb{Z}_{\geq 0}$, 
	\[
	\left|\frac{\partial^{2n}}{\partial x_2^{2n}} h(0,0)\right|\geq \frac{(2n)!}{2^{2n}} \cdot M_{2n}.
	\]
\end{enumerate}
\end{proposition}
\begin{proof}
{Let $h(x)$ denote the function}
\[
h(x):= \sum_{k\geq 1} \frac{m_k^2}{2^k \varphi(m_k)} \cdot \frac{1}{x^2_1+(m_kx_2)^2+1},
\]
where $m_k = M_{k+1}/M_k$ and {$\varphi = \varphi_M$} is the Ostrowski function. By Lemma \ref{lem: triv der formula 2var} and the chain rule, 
\begin{equation}\label{eq: h y der 1st estimate}
\left|\partial^{\alpha} h(x)\right|\leq {64} \cdot {8}^{|\alpha|+1}\alpha! \sum_{k\geq1} \frac{m_k^2}{2^k \varphi(m_k)} \cdot \frac{m_k^{\alpha_2}}{(x_1^2+(m_kx_2)^2+1)^{1+|\alpha|/2}},
\end{equation}
for every $\alpha \in \mathbb{Z}^2_{\geq 0}$. It follows from {the definition of $\varphi$} and the fact that $m_k\geq 1$ that
\[
\begin{aligned}
\left|\partial^{\alpha} h(x)\right|&\leq  \frac{{64} \cdot{8}^{|\alpha|+1}\alpha!}{(x_1^2+x_2^2+1)^{1+|\alpha|/2}} \sum_{k\geq1} \frac{m_k^{2+\alpha_2}}{2^k \varphi(m_k)}\\
&\leq \frac{{64} \cdot{8}^{|\alpha|+1}\alpha!}{(x_1^2+x_2^2+1)^{1+|\alpha|/2}} M_{\alpha_2},
\end{aligned}
\]
proving the upper-bound estimate (i). Next, from the Taylor expansion of $(1+z)^{-1}$, we get
\[
\frac{\partial^{2n}}{\partial x_2^{2n}} h(0,0)=(-1)^n (2n)! \sum_{k\geq 1} \frac{1}{2^{k}}\cdot \frac{m_k^{2n+2}}{\varphi(m_k)}\,,
\]
and, using Lemma \ref{lem:PropertyVarphiI}, we conclude that
\[
\begin{aligned}
\left| \frac{\partial^{2n}}{\partial x_2^{2n}} h(0,0) \right| &= (2n)! \left| \sum_{k\geq 1} \frac{1}{2^{k}}\cdot \frac{m_k^{2n+2}}{\varphi(m_k)} \right|\geq \frac{(2n)!}{2^{2n}}\cdot \frac{m_{2n}^{2n+2}}{\varphi(m_{2n})}  {=} \frac{(2n)!}{2^{2n}} \cdot M_{2n},
\end{aligned}
\]
proving the lower bound {estimate (ii)}. 
\end{proof} 

\noindent
Now, let $\rho \in (0,1)$ and let $q=(q,0)\in\mathbb{R}^2$. We consider the function
\[
f_{q,\rho}(x) := h\left( \frac{x}{\rho}-q \right),
\]
and we denote by $p = (\rho\cdot q,0)$ the associated centre point. 

\begin{remark}[{\it A priori} estimates on the building block]\label{lem: f estimates}
It follows from Proposition \ref{prop: h def} and the chain rule, that for every $\rho>0$ and $q \geq 1$, the function $f_{q,\rho}(x)$ satisfies the following estimates:
	\begin{enumerate}
		\item[(i)] for every $x\in\mathbb{R}^2$ and  $\alpha\in\mathbb{Z}_{\geq 0}^2$,
		\[
		| \partial^{\alpha} f_{q,\rho}(x)| \leq \frac{{64} \rho^2 \cdot {8}^{|\alpha|+1}\alpha ! M_{\alpha_2}}{\left(\Vert x- q\rho \Vert^2+\rho^2\right)^{1+|\alpha|/2}}\,	.
		\]
		\item[(ii)] for every $n \in \mathbb{Z}_{\geq 0}$, 
		\[
		|\partial_{x_2}^{2n} f_{q,\rho}(p)| \geq \frac{(2n)! \cdot M_{2n}}{4^n \cdot \rho^{2n}}\,.
		\]
	\end{enumerate}
\end{remark}

\noindent
Now, let us consider the blowing-up $\sigma: \IR^2 \to \IR^2$ (in polar coordinates),
\[
\sigma(r,\theta) = (r \cos\theta,\, r \sin\theta).
\]
Set $g_{\rho,q}:= f_{\rho,q}\circ \sigma (r,\theta)$. We summarize the main properties of $g_{\rho,q}$:

\begin{lemma}[{\it A priori} estimate on the blowing-up of the building block]\label{lem: g der estiamtes} Let $M=(M_{n})$ denote a log convex sequence starting with ${M_0 =} M_1=1$. There is a universal constant $C>0$ such that, for every $\rho\in (0,1)$, $q\geq 1$, $\alpha \in \mathbb{Z}_{\geq 0}^{2}$, $r>0$ and 
$\theta \in [-\pi,\pi]$,
	\[
\left| \partial^{\alpha} g_{q,\rho}(r,\theta)\right|\le C^{|\alpha|+1} \cdot \left( 1 + q\rho \right)^{\alpha_2}\cdot\alpha! \cdot M_{|\alpha|}.
\]
\end{lemma}
\begin{proof}
{From the definitions} of $h$ and $f_{q,\rho}$, we obtain
	\[
	\begin{aligned}
	 g_{q,\rho}(r,\theta)&=\sum_{k\geq 0} \frac{m_k^2}{2^k \varphi(m_k)} \cdot \frac{\rho^2}{(r\cos\theta-q\rho)^2+(m_kr\sin\theta)^2+\rho^2}\\
	 &=\sum_{k\geq 0} \frac{m_k^2}{2^k \varphi(m_k)} v_{q,m_k,\rho}(r,\theta).
	\end{aligned}
	\]
	where $v_{q,m_k,\rho}(r,\theta) = u_{q,m_k,\rho}\circ \sigma$ is the brick function. Now, by Lemma~\ref{lem: base function der in polar coor} with $m=m_k$, we get
 \[ 
 \left|  \partial^{\alpha} g_{q,\rho} \right| \leq  C^{|\alpha|+1}\alpha!\cdot \left( 1 + q\rho \right)^{\alpha_2}  \sum_{k\geq 0} \frac{m_k^2}{2^k \varphi(m_k)} m_k^{|\alpha|} ,
 \] 
and the result follows from {the definition of $\varphi$}.
\end{proof}

\subsection{A flat construction}\label{ssec:Flat}
Let $M$ be an increasing log convex sequence such that $M_0= 1$. Let $E:[0,\infty) \to [0,1)$ 
denote an increasing continuous function such that 
\begin{equation}\label{eq:E-Estimate}
E(0)=0 \quad\text{and}  \quad\lim_{r\to 0^+}\frac{E(r)}{r}=\infty.
\end{equation}
Consider the sequences
\begin{equation}\label{eq:seqs}
\rho_n:= M_n/M_{n+1}, \quad q_n:=(E(\rho_n)/\rho_n,0).
\end{equation} 
{Note that, whenever the class $C_M$ properly contains the class of analytic functions, the sequence $(\rho_n)$ (which is non-increasing)
tends to zero.} Let 
$\Lambda\subset 2\mathbb{N}$ be an unbounded set of indices such that $q_{\lambda} > 1$, for every $\lambda \in \Lambda$. We denote by $\Gamma$ the data $(\Lambda,\rho_n,q_n)$, and we consider the function
\[
F_{\Gamma}({x}) := \sum_{\lambda\in\Lambda} \frac{f_{q_\lambda,\rho_\lambda}(x)}{\varphi(\rho_{\lambda}^{-1})2^\lambda}.
\]
{We will provide conditions on} the data $\Gamma$ in terms of the asymptotic behaviour of the sequence
\[
\delta_\lambda := 
\mbox{dist}\left(q_\lambda \rho_\lambda\;,\;\{q_{\lambda^\prime}\rho_{\lambda^\prime}\}_{\lambda^\prime\in\Lambda\setminus\{\lambda\}}\right)
=\mbox{dist}\left(E(\rho_{\lambda})\;,\;\{E(\rho_{\lambda^{\prime}})\}_{\lambda^\prime\in\Lambda\setminus\{\lambda\}}\right).
\]

\begin{lemma}[Regularity of $F_{\Gamma}$]\label{lem: Cinfty}
The function $F_{\Gamma}$ {belongs to $C_{M^2}(\mathbb{R}^2)$}.
\end{lemma}
\begin{proof}

{By the definition of $\varphi$} and Remark \ref{lem: f estimates}, {if  
{$x\in \IR^2$}, then}
\begin{align*}
\left|\sum_{\lambda\in\Lambda} \frac{\partial^{\alpha}f_{q_\lambda,\rho_\lambda}(x)}{2^\lambda \varphi(\rho_{\lambda}^{-1})}\right| &\leq  {64}\cdot {8}^{|\alpha|+1}\alpha! M_{|\alpha|}\sum_{\lambda\in\Lambda}\frac{\rho_{\lambda}^2}{2^\lambda \varphi(\rho_{\lambda}^{-1})\rho_\lambda^{|\alpha|+2}}\\
&\leq {8}^{|\alpha|+3}\alpha! M_{|\alpha|}M_{|\alpha|}< \infty.
\end{align*}
It follows from the Weierstrass M-test that $F_{\Gamma}$ is $C^{\infty}$ at $\mathbb{R}^2$, and the 
bounds above guarantee that it is in {$C_{M^2}$}.
\end{proof}

\noindent
Now, {consider the blowing-up $\s(r,\theta) = (r \cos\theta, r \sin\theta)$,} and set $G_{\Gamma}:= F_{\Gamma}\circ \sigma (r,\theta)$. We obtain the following estimates.

\begin{lemma}[Estimates of $ F_{\Gamma} $ after blowing up]\label{lem: G upper bound}
The function $G_{\Gamma}$ {belongs to} $C_{M}(\mathbb{R}^2)$.
\end{lemma}
\begin{proof}
{From the definitions of $F_\Gamma$ and} $G_{\Gamma}$, we get
\[ 
G_\Gamma(r,\theta)=\sum_{\lambda\in\Lambda}\frac{g_{q_\lambda,\rho_\lambda}(r,\theta)}{2^\lambda \varphi(\rho_{\lambda}^{-1})}, 
\]
so it follows from Lemma \ref{lem: g der estiamtes} that
\[  
\left|\partial^\alpha G_{\Gamma}(r,\theta)\right|  \leq C^{|\alpha|+1}\alpha! M_{|\alpha|}\sum_{\lambda\in\Lambda}\frac{1}{2^\lambda \varphi(\rho_{\lambda}^{-1})}\left( 1 + q_\lambda\rho_\lambda \right)^{\alpha_2} .
\]
Since $q_\lambda\rho_\lambda=E(\rho_\lambda)$, which is bounded by $1$, and $\varphi(\rho^{-1}_{\lambda}) \geq 1$ (because $\rho_{\lambda}<1$), we obtain
\[  
\left|\partial^\alpha G_R(r,\theta)\right|  \leq (2C)^{|\alpha|+1}\alpha! M_{|\alpha|}. \qedhere
\]
\end{proof}

\begin{lemma}[Lower estimates on $F_{\Gamma}$]\label{lem: upper bound derivative}
Suppose that the class ${C}_M$ is closed {under} differentiation {and properly contains the analytic functions}. Then
there is a constant $B>0$ (depending only on $M$) {which satisfies the following property: Assume} 
there exists $\lambda_0\in \Lambda$ such that $\delta_\lambda \geq  B M_{\lambda}^{-1/\lambda}$ when $\lambda\geq \lambda_0$ and 
$\lambda\in \Lambda \subset 2\mathbb{N}$.
Then there exists an infinite sequence of points $x_\lambda\to 0$ and a constant $\epsilon>0$ such that 
\[
|\partial^\lambda_{{x_2}} F_\Gamma(x_\lambda)| \geq \frac{\epsilon^{\lambda}\lambda!M_\lambda M_\lambda}{4^\lambda}, \quad \text{ for }\lambda \geq \lambda_0 \text{ and }\lambda \in \Lambda \subset 2\mathbb{N}.
\]
\end{lemma}
\begin{proof}
Consider the sequence of points $x_n=\rho_n q_n$, which converges to the origin {since $(\rho_n)$ tends to zero}. Note that
\[
|\partial^\lambda_{x_2} F_{\Gamma}(x_\lambda)| \geq \frac{1}{2^\lambda\varphi(\rho_{\lambda}^{-1})}|\partial^\lambda_{x_2} f_{q_\lambda,\rho_{\lambda}}({x}_\lambda)| 
- \left| {\sum_{\lambda^\prime\neq\lambda} \frac{\partial^{\lambda'}_{x_2} f_{q_{\lambda'},\rho_{\lambda'}}({x}_{\lambda'})}{2^{\lambda'}\varphi(\rho_{\lambda^\prime}^{-1})} } \right|.
\]
If follows from Remark \ref{lem: f estimates} and Lemma \ref{lem:PropertyVarphiI} that
\[
\begin{aligned}
|\partial^\lambda_{{x_2}} F_{\Gamma}(x_\lambda)| &\geq \lambda!\frac{M_\lambda}{(2\rho_\lambda)^\lambda \varphi(\rho_{\lambda}^{-1})} - \frac{\lambda! M_{\lambda} {8}^{\lambda+3}}{\delta_\lambda^{\lambda}} \sum_{\lambda^\prime\neq\lambda} \frac{1}{2^{\lambda^\prime}}\\
& \geq\frac{\lambda!\rho_{\lambda}^2M_\lambda M_\lambda}{{2^{\lambda}}} - \frac{\lambda! M_{\lambda}{8}^{\lambda+3}}{\delta_\lambda^{\lambda}} \,.
\end{aligned}
\]

Now, since the class is closed under differentiation, it follows from the criterion \eqref{eq:ClosedByDiff} that there exists $\epsilon \in (0,1]$ such that $\rho_{\lambda}^2 M_\lambda M_{\lambda} \geq \epsilon^{\lambda} M_\lambda M_{\lambda}$. Let $B= {8^5}\epsilon^{-1}$. Under the hypothesis of the lemma, for $\lambda\geq \lambda_0$,
\[
|\partial^\lambda_{x_2} F_{\Gamma}(x_\lambda)| \geq \frac{\epsilon^{\lambda}\lambda!M_\lambda M_\lambda}{{2^{\lambda}}} - \frac{\epsilon^{\lambda}\lambda!M_\lambda M_\lambda}{{8}^\lambda}\geq \frac{\epsilon^{\lambda}\lambda!M_\lambda M_\lambda}{4^\lambda}\,. \qedhere
\]
\end{proof}

\subsection{Proofs of the main Theorems \ref{thm:Main} and \ref{thm:Main2}}\label{ssec:proofMain}

The main technical result of this section is the following {lemma}.

\begin{lemma}\label{cl:Main}
Let $M$ denote a log convex sequence with $M_0=1$, such that $C_M$ is a Denjoy-Carleman class closed 
under differentiation, {which properly contains the class of analytic functions}. Let $\Lambda' \subset \mathbb{N}$ be an infinite {set of indices such that}
\[
\inf_{{k \in \Lambda'}}\left\{  \left( \frac{M_{k}^2}{M_{2k}} \right)^{1/k}   \right\} >0.
\]
Then there exist {an infinite} subset $\Lambda \subset \Lambda'$, a constant $\mathcal{K} \in \mathbb{R}_{> 0}$, a sequence of points $(x_\lambda)_{\lambda \in \Lambda} \in \mathbb{R}^2$ {tending to} $(0,0)$, and a function $F \in C_{M^{(2)}}(\mathbb{R}^2)$, such that $F \circ \sigma \in C_{M}(\mathbb{R}^2)$ and
\[
|\partial^\lambda_{x_2} F(x_\lambda)| \geq \mathcal{K}^{\lambda} \lambda! M_{2\lambda}, \quad 
\text{ for all }\, \lambda \in \Lambda.
\]
\end{lemma}

\begin{proof} 
{Let $E$ denote a function as in \eqref{eq:E-Estimate}, and let $\rho_n$ and $q_n$ denote
the sequences given by \eqref{eq:seqs}.} Let $B>0$ be the constant {(depending only on $M$)}
given by Lemma \ref{lem: upper bound derivative}. The proof {of the lemma is divided into two steps:}

\medskip
\noindent
\emph{Step I.} Suppose that $\Lambda' \cap 2 \mathbb{N}$ is an infinite set. We {construct  
$\Lambda$ satisfying the assumption} in Lemma \ref{lem: upper bound derivative}: {By the hypothesis
above} and equation \eqref{eq:ComparisonConditions}, 
\begin{equation}\label{eq:NecessaryConditionQuasi}
{\xi :=} \inf_{\lambda \in \Lambda^{\prime}} \left\{\frac{\left(M_\lambda\right)^{{1+ 1/\lambda}}}{M_{\lambda+1}} \right\} >0,\quad \text{and} \quad  \inf_{\lambda \in \Lambda^{\prime}} \left\{ \left(\frac{M_{\lambda}^2}{M_{2\lambda}}\right)^{{1/\lambda}} \right\} >0.
\end{equation}
Since the sequence $(\rho_n)_{n\ge 1}$ is {non-increasing and tends} to zero, we can choose {infinite}
$\Lambda\subset\Lambda^\prime$ {sparse enough} that $E(\rho_\lambda)>2E(\rho_{\lambda^\prime})$, for any $\lambda,
 \lambda^\prime\in\Lambda$ such that $\lambda<\lambda^\prime$. We fix such $\Lambda$, and {will
prove that it satisfies the assumption} in Lemma \ref{lem: upper bound derivative}. Indeed, for fixed $\lambda\in \Lambda$, let $\lambda_-$ and $\lambda_+$ denote the smaller and larger neighbours of $\lambda$ (respectively); i.e., 
 \[ 
 \lambda_-:=\max \left(\Lambda \cap (0,\lambda)\right),\quad \lambda_+:=\min \left(\Lambda \cap (\lambda,+\infty)\right). 
 \]
 {Let 
 \[
  \delta_\lambda :=\min\{E(\rho_{\lambda_-})-E(\rho_\lambda)\;,\; E(\rho_\lambda)-E(\rho_{\lambda_+}) \}.
 \]
 Since $(\rho_n)_{n\ge 1}$ is {non-increasing},
 \[ 
\delta_\lambda \geq \frac{1}{2} \min\{E(\rho_{\lambda_-}), E(\rho_\lambda) \}=\frac{E(\rho_\lambda)}{2}. 
 \]}
{By} \eqref{eq:E-Estimate} and the first {condition} in \eqref{eq:NecessaryConditionQuasi}, there exists $\lambda_0$ such that
\[
\delta_{\lambda} \geq  \frac{E(\rho_\lambda)}{2} > B \xi^{-1}\rho_\lambda =  B \xi^{-1} \frac{M_\lambda}{M_{\lambda+1}} \geq \frac{B}{M_{\lambda}^{1/\lambda}},  \quad \text{ for } \lambda\geq \lambda_0,
\]
showing that the the {assumption} in Lemma \ref{lem: upper bound derivative} is satisfied. 

Finally, consider the function $F_{\Gamma}$ {determined by the}
data $(\Lambda,\rho_n,q_n)$. By Lemmas \ref{lem: Cinfty} and \ref{lem: G upper bound},
$F \in C_{M^2}(\mathbb{R}^2) \subset C_{M^{(2)}}(\mathbb{R}^2)$ and $F_{\Gamma} \circ \sigma \in C_M(\mathbb{R}^2)$. By Lemma \ref{lem: upper bound derivative} and the the second {condition} in \eqref{eq:NecessaryConditionQuasi}, there exists $\epsilon>0$ and a sequence of points $x_{\lambda} \to 0$ such that
\[
|\partial^\lambda_y F_\Gamma(x_\lambda)| \geq \frac{\epsilon^{\lambda}\lambda!M_\lambda M_\lambda}{4^\lambda} \geq \frac{\epsilon^{\lambda}\lambda!M_{2\lambda}}{8^\lambda}, \quad \text{ for }\lambda \geq \lambda_0,
\]
as we wanted to prove.

\medskip
\noindent
\emph{Step II.} Suppose that $\Lambda' \cap 2 \mathbb{N}$ is a finite set. Consider the set of indexes $\widetilde{\Lambda}' = \{\lambda + 1:  \lambda \in \Lambda'\}$ and note that $\widetilde{\Lambda}' \cap 2\mathbb{N}$ is infinite. It follows from the Step I that there exist {an infinite}
subset $\widetilde{\Lambda} \subset \widetilde{\Lambda}'$, a constant $\widetilde{\mathcal{K}} \in \mathbb{R}_{> 0}$, a sequence of points $(x_{\widetilde{\lambda}})_{\widetilde{\lambda} \in \widetilde{\Lambda}} \in \mathbb{R}^2$, and a function $\widetilde{F} \in C_{M^{(2)}}(\mathbb{R}^2)$,
such that $F \circ \sigma \in C_{M}(\mathbb{R}^2)$ and
\[
|\partial^\lambda_{x_2} F(x_{\widetilde{\lambda}})| \geq \widetilde{\mathcal{K}}^{\lambda} \lambda! M_{2\lambda}, \quad \text{ for all }\, \lambda \in \widetilde{\Lambda}.
\]
Let $\Lambda:=\{\lambda-1; \, \lambda \in \widetilde{\Lambda}\} \subset \Lambda'$ and set $F := \partial_{x_2}\widetilde{F}$. Since $C_{M^{(2)}}$ is closed under differentiation (cf. \S\ref{subsec:shift}), it follows that $F \in C_{M^{(2)}}(\mathbb{R}^2)$. Note that
\[
F\circ \sigma = \frac{1}{r} \left( r\sin\theta\,\partial_r(\widetilde{F}\circ \sigma) + \cos\theta\,\partial_{\theta}(\widetilde{F}\circ \sigma) \right);
\]
it follows that $F\circ \sigma \in C_M$, since $\widetilde{F}\circ \sigma \in C_M$ and $C_M$ is closed under differentiation and division by a monomial. Finally, note that the point $x_{\lambda}$ is well-defined for every 
$\lambda \in \Lambda$, and
\[
|\partial^\lambda_{x_2} F(x_\lambda)| = |\partial^{\lambda+1}_{x_2} \widetilde{F}(x_{\widetilde{\lambda}})|   \geq \widetilde{\mathcal{K}}^{\lambda+1} (\lambda+1)! M_{2\lambda+2} \geq \mathcal{K}^{\lambda} \lambda! M_{2\lambda}.
\]
for some $\mathcal{K}>0$; {this completes the proof of the second step and of the lemma}.
\end{proof}

We {can now prove the main theorems.}

\begin{proof}[Proof of Theorem \ref{thm:Main}] {The result is trivial when $C_M$ is the class of analytic functions, so suppose that $C_M$ properly contains the analytic function and consider} $C_{N} \subsetneq C_{M^{(2)}}$. 
{By hypothesis,  $C_{M^2} = C_{M^{(2)}}$.}
{By the criterion of \S\ref{ssec:Comparison}(b)} applied to $C_N$ and $C_{M^2}$,
there exists {an {infinite} subset $\Lambda' \subset \mathbb{N}$ such that}
\[
{\lim_{k \to \infty,\, k \in \Lambda'}}\left\{  \left( \frac{N_{k}}{M_{2k}} \right)^{1/k}   \right\} = 0 \quad \text{ and } \quad \inf_{{k \in \Lambda'}}\left\{  \left( \frac{M_{k}^2}{M_{2k}} \right)^{1/k}   \right\} >0 .
\]
The result follows easily from Lemma \ref{cl:Main}.
\end{proof}

\begin{proof}[Proof of Theorem \ref{thm:Main2}] {The result is trivial when $C_M$ is the class of analytic functions, so suppose that $C_M$ properly contains the analytic function and consider} $C_{N} \subsetneq C_{M^{(2)}}$ such that $\lim_{k\to \infty} \left(N_k/M_{2k}\right)^{1/k} =0$.
By Lemma \ref{lem: extrareg2}, there exists {an infinite subset of indices
$\Lambda'\subset \mathbb{N}$ such that}
\[
\inf_{{k \in \Lambda'}}\left\{  \left( \frac{M_{k}^2}{M_{2k}} \right)^{1/k}   \right\} >0,
\]
{whereas} $\inf_{{k \in \Lambda'}}\left\{  \left( N_{k}/M_{2k} \right)^{1/k}   \right\} = 0$. The result {follows} easily from Lemma \ref{cl:Main}.
\end{proof}

\section{Proof of Proposition \ref{prop:Difference}}\label{sec:ProofProp}
\noindent
Fix a positive increasing sequence $(v_k)$ such that
$$
v_k\to\infty,\quad \frac{v_{k+1}}{v_k}\uparrow 1\quad \text{and}\quad  v_{2k}/v_k\uparrow\infty;
$$
for example, we may take $v_k = \prod_{j\leq k}(1+1/\sqrt{j})$, for every $k\in \mathbb{N}$. 
Let {$\Lambda=(\lambda_n)$  be an increasing sequence of nonnegative} even integers
with $\lambda_0=0$, such that 
$$
|\Lambda\cap2\cdot\Lambda|=\infty,\quad\text{and}\quad \sum_{n>0}\frac{\log \lambda_{n+1}-\log\lambda_{n}}{ v_{n}^2}=\infty.
$$
{For every $k\in\IN$ such that} $\lambda_n\leq k<\lambda_{n+1}$, set $L_k=v_n$.  Finally, set 
$M_k=L_k^k$. {Let us verify that $C_M$ is the class sought:}

\medskip
\noindent
\emph{$M$ is a log convex sequence.} Indeed, let $a_k= \log M_{k}-\log M_{k-1}$. By direct computation,
\[
\begin{aligned}
a_{k} &=  \log(v_{n}), & \text{ if } &\lambda_n< k\leq \lambda_n;\\
a_{k} &=  \log(v_{n})+(k-1)\log\left(\frac{v_n}{v_{n-1}}\right),  & \text{ if } &k=\lambda_n.
\end{aligned}
\]
{This implies that $(a_k)$ is an increasing sequence, and the result follows}.

\medskip
\noindent
\emph{$C_M$ is closed under differentiation.} Let 
\[
b_k=\left(\frac{M_{k+1}}{M_k}\right)^{1/k}=\, \frac{L_{k+1}}{L_k}\cdot L_{k+1}^{1/k}.
\]
{It is enough} to show that $b_k$ is bounded from above. By the inequality of arithmetic and geometric means, $b_k\leq \left(L_{k+1} / L_k\right)^2$, and direct computation shows that
\[
\begin{aligned}
\left(\frac{L_{k+1}}{L_k}\right)^2 &= 1,   & \text{ if } &\lambda_n< k+1\leq \lambda_n;\\[.5em]
\left(\frac{L_{k+1}}{L_k}\right)^2 &=  \left(\frac{v_{n}}{v_{n-1}}\right)^2,  & \text{ if } &k+1=\lambda_n,
\end{aligned}
\]
so there is indeed an upper bound.

\medskip
\noindent
\emph{$C_{M^{(2)}}$ is quasianalytic.} Since $L_n$ is non-decreasing,
\[ 
\sum_{n\geq 1}\frac{1}{n L_{2n}^2}\ge\sum_{n\geq 1}\frac{1}{ L_{2^{n+1}}^2} = \sum_{n\geq 1}\,\sum_{\lambda_n\leq 2^{k+1}<\lambda_{n+1}}\frac{1}{ L_{n}^2}\geq\sum_{n\geq1} \frac{\log_2 \lambda_{n+1}-\log_2 \lambda_{n}}{v^2_{n}}.  
\]
By assumption,  the sum on the right hand side is divergent.

\medskip
\noindent
\emph{$C_{M^2}\subsetneq C_{M^{(2)}}$.} Observe that 
\[
\left( \frac{M_k^2}{M_{2k}}\right)^{1/k}= \left(\frac{L_k}{L_{2k}} \right)^2. 
\]

Since $\Lambda\cap2\Lambda$ is an infinite and $v_k/v_{2k}\to 0$, {we see that} 
zero is a partial limit of the sequence above. 
By construction, the partial limits of this sequence are 0 and 1. In particular, 
\[
\limsup_{k\to \infty} \frac{M_n^2}{M_{2n}} \leq1 \quad \text{and} \quad  \liminf_{k\to \infty} \frac{M_n^2}{M_{2n}} =0. 
\]
{The assertion follows from the criterion \S\ref{ssec:Comparison}(b).
\qed}

\end{document}